\documentclass[a4paper]{article}
\usepackage[english]{babel}
\usepackage[cp1251]{inputenc}
\pdfoutput=1
\usepackage[colorlinks,linkcolor=blue,citecolor=blue,urlcolor=blue]{hyperref}
\usepackage{latexsym,amsfonts,amssymb,amsthm,amsmath,amsbsy,graphicx}
\newtheorem{thm}{Theorem}[section]
\newtheorem{lem}{Lemma}[section]

\theoremstyle{definition}

\begin{document}
\thispagestyle{empty}
\bigskip
 {\noindent \Large\bf\sc  two-sided boundary functionals for kou
 process}\footnotemark[1]\footnotetext{This is an electronic reprint
  of the original article published in Visnyk
of Dnipropetrovs'k National University. This reprint is the translation
from Ukrainian and has minor structural changes.}

\bigskip
Ie.~Karnaukh
 \footnote{Department of Statistics and Probability Theory,
Dnipropetrovsk Oles Honchar National University,\\ 72, Gagarina pr.,
Dnipropetrovsk 49010, Ukraine.
 \href{mailto:ievgen.karnaukh@gmail.com}{ievgen.karnaukh@gmail.com}
 }
\begin{center}
\bigskip
\begin{quotation}
\noindent  {\small In this paper the moment generating function of two-sided boundary functionals for
double exponential jump diffusion processes are treated.}
\end{quotation}
\end{center}
\bigskip

Using the results of \cite{Kou2003,Husak2011engl} we investigate
the so called two-sided boundary functionals for a Kou process, which is a L\'evy
process with bounded variation positive and negative jumps that have
exponential distribution. To find the representation of moment
generating functions for the functionals we apply the factorization method
(see for instance \cite{Husak2011engl}), by which the solution
of the corresponding integral equation can be determined in terms of the
distributions of killed extrema. For another methods we refer the reader to
\cite{Kadankov2005engl, Bratiychuk1990} (for the method of successive iterations) and
 \cite{Bratijchuk1984engl} (for the resolvent method).

\section{Kou process and its approximating process}
Consider the stochastic process
\begin{equation}
\xi\left(t\right)=at+\sigma W\left(t\right)+S(t),\;\xi\left(0\right)=0,\; t\geq0,\label{eq:1.1}
\end{equation}
where $a$ is a finite constant, $\sigma>0$, $W(t)$ is a standard
Wiener process, $S(t)$ is a Poisson process with intensity $\lambda>0$ and
dual exponential jumps with density
\begin{equation}
f\left(x\right)=pce^{-cx}I_{\left\{ x\geq0\right\} }+qbe^{bx}I_{\left\{ x<0\right\} },\left(c,b,p,q>0,p+q=1\right).\label{eq:1.2}
\end{equation}
We call process (\ref{eq:1.1}) a Kou process.

Let $\theta_{s}$ be an exponential random variable with parameter
$s>0$ independent of $\xi(t)$, then the moment generating function
(m.g.f.) of a Kou process killed at rate $s$ is as follows
\[
\mathsf{E}e^{r\xi\left(\theta_{s}\right)}=\int_{0}^{\infty}se^{-st}\mathsf{E}e^{r\xi\left(t\right)}dt=\frac{s}{s-k(r)},\;\mathrm{Re}[r]=0,
\]
where $k\left(r\right)$ is the cumulant function of the
form
\begin{equation}
k(r)=ar+r^{2}\frac{\sigma^{2}}{2}+\lambda r\left(\frac{p}{c-r}-\frac{q}{b+r}\right).\label{eq:1.3}
\end{equation}

To find the distribution of a functional for a L\'evy process important
role plays the approximating method (see \cite{Bratiychuk1990} and references
therein). By this method, we construct a prelimit process and deduce
an integro-differential equation for the m.g.f. of the corresponding functional. Then solving
the equation and passing to the limit yields the solution for the initial
process.

The prelimit Kou process is given by
\[
\xi_{n}\left(t\right)=a_{n}t+S_{n}(t),
\]
where $a_{n}=a+3n\sigma^{2}/2,$ $S_{n}\left(t\right)$ is a compound
Poisson process with intensity $\lambda_{n}=\lambda p+3n^{2}\sigma^{2}+\lambda qe^{-b/n}$
and the density of jumps
\[
f_{n}\left(x\right)=\begin{cases}
p_{1}\left(n\right)ce^{-cx}, & x\geq0;\\
p_{2}\left(n\right)n, & -\frac{1}{n}\leq x<0;\\
p_{3}\left(n\right)be^{b\left(x+1/n\right)}, & x\leq-\frac{1}{n},
\end{cases}
\]
where $p_{1}\left(n\right)=\lambda p/\lambda_{n},p_{2n}\left(n\right)=3n^{2}\sigma^{2}/\lambda_{n},p_{3}\left(n\right)=e^{-b/n}\lambda q/\lambda_{n}$.
That is, the prelimit process is a Poisson process with positive
drift (possibly for large enough $n$), exponentially distributed
positive jumps and negative jumps that have a mixture of shifted exponential
and uniform distributions.

Denote the event $A_{n}\left(T\right)=$ $\left\{ \omega:\sup_{t\leq T}\left|\xi\left(t\right)-\xi_{n}\left(t\right)\right|>1/\sqrt{n}\right\} $,
then by Kolmogorov's inequality: $\mathsf{{P}}\left(A_{n}\left(T\right)\right)\leq Tn\frac{2e^{-b/n}}{b^{2}}\left(e^{b/n}-1-b/n-\frac{b^{2}}{2n^{2}}\right)\sim\frac{1}{n^{2}}$ (for more general case see \cite{Bratiychuk1990}).
Therefore, for any $T>0$ the series $\sum_{n\geq1}\mathsf{{P}}\left(A_{n}\left(T\right)\right)$
converges and using the Borel--Cantelli lemma we obtain $\mathsf{{P}}\left\{ \bigcup_{n\geq1}\bigcap_{k\geq n}A_{k}\left(T\right)\right\} =1$.
Thus
\[
\mathsf{{P}}\left\{ \bigcap_{T>0}\lim_{n\rightarrow\infty}\sup_{0\leq t\leq T}\left|\xi\left(t\right)-\xi_{n}\left(t\right)\right|=0\right\} =1.
\]

For the cumulant function of prelimit process we have
\begin{multline*}
k_{n}\left(r\right)=a_{n}r+\int_{-\infty}^{\infty}\left(e^{rx}-1\right)\lambda_{n}f_{n}\left(x\right)dx=ar+\lambda qe^{-b/n}\left(\frac{b}{b+r}e^{r/n}-1\right)+\\
+\lambda p\left(\frac{c}{c-r}-1\right)-\frac{3n^{3}\sigma^{2}}{r}\left(e^{-r/n}-1+r/n-\frac{1}{2}\left(r/n\right)^{2}\right)\underset{n\rightarrow\infty}{\longrightarrow}k\left(r\right).
\end{multline*}
By \cite{Gihman1975}, finite-dimensional distributions of the prelimit process tend to ones of the Kou process and for some constant $C$, which possibly depends
on $\epsilon>0$
\[
\overline{\lim}_{n\rightarrow\infty}\sup_{\left|t_{1}-t_{2}\right|\leq h}\mathsf{{P}}\left\{ \left|\xi_{n}\left(t_{1}\right)-\xi_{n}\left(t_{2}\right)\right|>\epsilon\right\} \leq Ch.
\]
Moreover, let $f_{T}\left(x\left(\cdot\right)\right)$ be
a functional determined on the Skorokhod space ${D}_{\left[0,T\right]}\left({R}\right)$.
If $f_{T}\left(x\left(\cdot\right)\right)$ almost everywhere continuous in the Skorokhod topology, then the distribution of $f_{T}\left(\xi_{n}\left(\cdot\right)\right)$
tends to the distribution of $f_{T}\left(\xi\left(\cdot\right)\right)$.
By \cite{Bratiychuk1990}, the class of such functionals
comprises supremum, infimum, overshoots and two-sided boundary functionals.

The basic functionals are the extrema of the process:
\[
\xi^{+}\left(t\right)=\sup_{u\leq t}\xi\left(u\right),\,\xi^{-}\left(t\right)=\inf_{u\leq t}\xi\left(u\right),
\]
for which the factorization identity
(the Spitzer-Rogozin identity) is hold
\[
\mathsf{E}e^{r\xi\left(\theta_{s}\right)}=
\mathsf{E}e^{r\xi^{+}\left(\theta_{s}\right)}\mathsf{E}e^{r\xi^{-}\left(\theta_{s}\right)},\mathrm{Re}[r]=0,
\]
where factors $\mathsf{E}e^{r\xi^{\pm}\left(\theta_{s}\right)}$ have the
analytic continuation into the half-plane $\pm\mathrm{Re}[r]\geq0$, respectively. The distribution of the other mentioned functionals can be represented
in terms of the distribution of extrema.

Following \cite{Kou2003}, the cumulant equation $k\left(r\right)=s$
for a Kou process $\xi\left(t\right)$ has exactly two positive roots $\rho_{1,2}\left(s\right):$
$\rho_{1}\left(s\right)<c<\rho_{2}\left(s\right)$ and two negative
roots $-r_{1,2}\left(s\right):$ $r_{1}\left(s\right)<b<r_{2}\left(s\right)$,
which define the destity functions of killed extrema and the m.g.f. of
overshoot functionals.
\begin{lem}{\cite{Kou2003}} \label{lem:1-1}
For a Kou process the density function
of killed supremum and infimum can be represented as the sum of exponents:
\begin{gather}
P'_{+}\left(s,x\right)=\frac{\partial}{\partial x}\mathsf{P}\left\{ \xi^{+}\left(\theta_{s}\right)<x\right\} =A_{1}^{+}e^{-\rho_{1}\left(s\right)x}+A_{2}^{+}e^{-\rho_{2}\left(s\right)x},x>0,\label{eq:2.5a}\\
P'_{-}\left(s,x\right)=\frac{\partial}{\partial x}\mathsf{P}\left\{ \xi^{-}\left(\theta_{s}\right)<x\right\} =A_{1}^{-}e^{r_{1}\left(s\right)x}+A_{2}^{-}e^{r_{2}\left(s\right)x},x<0,\label{eq:2.5b}
\end{gather}
where $A_{i}^{+}=(-1)^{i-1}\frac{c-\rho_{i}\left(s\right)}{c}\frac{\rho_{1}\left(s\right)\rho_{2}\left(s\right)}{\rho_{2}\left(s\right)-\rho_{1}\left(s\right)}$
and $A_{i}^{-}=(-1)^{i-1}\frac{b-r_{i}\left(s\right)}{b}\frac{r_{1}\left(s\right)r_{2}\left(s\right)}{r_{2}\left(s\right)-r_{1}\left(s\right)}$,
$i=1,2$.

The first passage time, $\tau^{+}\left(x\right)=\inf\left\{ t\geq0:\xi\left(t\right)>x\right\} ,x\geq0$,
and the overshoot, $\gamma^{+}\left(x\right)=\xi\left(\tau^{+}\left(x\right)\right)-x$,
are conditionally independent and the overshoot is conditionally memoryless
given that $\{\gamma^{+}\left(x\right)>0\}$:
\begin{gather}
\mathsf{E}\left[e^{-s\tau^{+}\left(x\right)},\gamma^{+}\left(x\right)=0,\tau^{+}\left(x\right)<\infty\right]=\frac{c-\rho_{1}\left(s\right)}{\rho_{2}\left(s\right)-\rho_{1}\left(s\right)}e^{-\rho_{1}\left(s\right)x}+\frac{\rho_{2}\left(s\right)-c}{\rho_{2}\left(s\right)-\rho_{1}\left(s\right)}e^{-\rho_{2}\left(s\right)x},\label{eq:2.6a}\\
\mathsf{E}\left[e^{-s\tau^{+}\left(x\right)-u\gamma^{+}\left(x\right)},\gamma^{+}\left(x\right)>0,\tau^{+}\left(x\right)<\infty\right]=\nonumber \\
=\frac{\left(c-\rho_{1}\left(s\right)\right)\left(\rho_{2}\left(s\right)-c\right)}{c\left(\rho_{2}\left(s\right)-\rho_{1}\left(s\right)\right)}\left(e^{-\rho_{1}\left(s\right)x}-e^{-\rho_{2}\left(s\right)x}\right)\frac{c}{c+u}.\label{eq:2.6b}
\end{gather}
\end{lem}
Using formulas (\ref{eq:2.5a}) -- (\ref{eq:2.6b}) we can deduce
the relations for the m.g.f. of functionals connected with the
first exit time from a fixed interval.

\section{Two-sided boundary functionals }
Define the first exit time from the interval $\left(x-T,x\right)$
$\left(0<x<T\right)$ by the Kou process as
\[
\tau\left(x,T\right)=\inf\left\{ t\geq0:\xi\left(t\right)\notin\left(x-T,x\right)\right\} ,
\]
and assume that $\tau\left(x,T\right)=0$ for $x\notin(0,T)$. Introduce
the events of the exit from the interval through the
upper and lower bounds
\[
A_{+}\left(x\right)=\left\{ \omega:\xi\left(\tau\left(x,T\right)\right)\geq x\right\} ,\; A_{-}\left(x\right)=\left\{ \omega:\xi\left(\tau\left(x,T\right)\right)\leq x-T\right\} ,
\]
 and the m.g.f. of the exit time through the corresponding bound as
\[
Q^{T}\left(s,x\right)=\mathsf{{E}}\left[e^{-s\tau\left(x,T\right)},A_{+}\left(x\right)\right],\; Q_{T}\left(s,x\right)=\mathsf{{E}}\left[e^{-s\tau\left(x,T\right)},A_{-}\left(x\right)\right].
\]
The m.g.f. of exit time for the Kou process we derive as the limit of the
corresponding m.g.f. for the prelimit process. For simplicity, we
suppress the explicit dependence on $n$ in the notation of parameters
of the prelimit process.

\subsection{\label{sub:M.g.f. of the exit} M.g.f. for the moment of exit from
an interval}
Consider the next stochastic relation for $\tau\left(x,T\right)$
on $A_{+}\left(x\right)$. Let $\zeta,\eta$ denote the moment and value of the first jump of $\xi_{n}\left(t\right)$,
respectively (see Fig.~\ref{pic.prelim}), then
\[
\tau\left(x,T\right)\dot{=}\begin{cases}
x/a, & a\zeta>x;\\
\zeta+\tau\left(x-a\zeta-\eta,T\right), & a\zeta<x,\, x-T<a\zeta+\eta<x;\\
\zeta, & a\zeta<x,\, a\zeta+\eta\geq x.
\end{cases}
\]
\begin{figure}
\centering \includegraphics[scale=0.8]{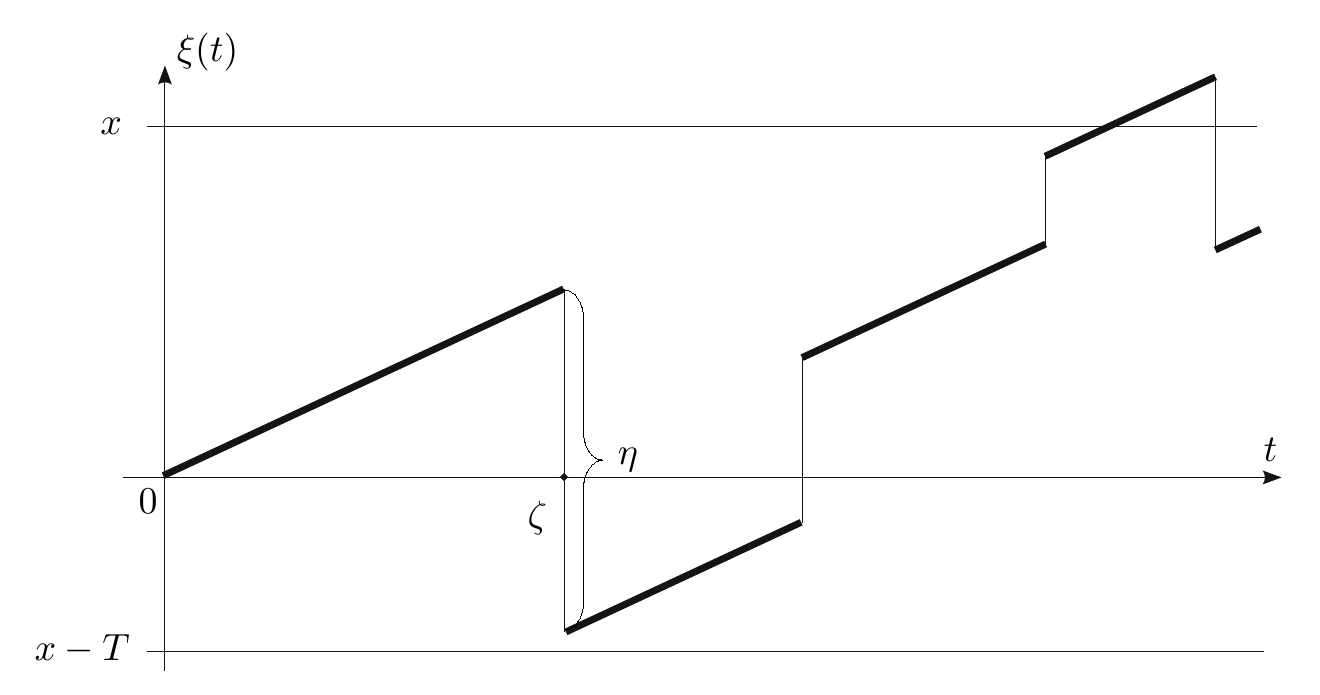} \caption{A path of the prelimit process}\label{pic.prelim}
\end{figure}

Using this stochastic relation and strong Markov property of $\xi_{n}\left(t\right)$
we deduce
\[
\begin{alignedat}{1}Q^{T}\left(s,x\right)= & \mathsf{{E}}\left[e^{-sx/a},a\zeta>x\right]+\mathsf{{E}}\left[e^{-s\left(\zeta+\tau\left(x-a\zeta-\eta,T\right)\right)},a\zeta<x,\, x-T<a\zeta+\eta<x\right]+\\
 & +\mathsf{{E}}\left[e^{-s\zeta},a\zeta<x,\, a\zeta+\eta\geq x\right].
\end{alignedat}
\]
Taking into account that $\zeta$ have exponential distribution with
parameter $\lambda$, and $\eta$ have density function $f\left(x\right)$
we obtain
\[
\begin{alignedat}{1}Q^{T}\left(s,x\right)= & e^{-\left(s+\lambda\right)\frac{x}{a}}+\int_{0}^{x/a}\lambda e^{-\left(s+\lambda\right)y}\int_{x-ay-T}^{x-ay}Q^{T}\left(s,x-ay-z\right)f\left(z\right)dzdy+\\
 & +\int_{0}^{x/a}\lambda e^{-\left(s+\lambda\right)y}\int_{x-ay}^{\infty}f\left(z\right)dzdy.
\end{alignedat}
\]
Combining this equation with boundary conditions
\[
Q^{T}\left(s,x\right)=\begin{cases}
1, & x\leq0,\\
0, & x\geq T,
\end{cases}
\]
yields
\[
Q^{T}\left(s,x\right)=e^{-\left(s+\lambda\right)\frac{x}{a}}+\frac{\lambda}{a}\int_{0}^{x}e^{-\left(s+\lambda\right)\frac{x-y}{a}}\int_{-\infty}^{\infty}Q^{T}\left(s,y-z\right)f\left(z\right)dzdy.
\]
Differentiating the last equation with respect to $x$ $\left(0<x<T\right)$ gives
\[
a\frac{\partial}{\partial x}Q^{T}\left(s,x\right)=-\left(s+\lambda\right)Q^{T}\left(s,x\right)+\lambda\int_{-\infty}^{\infty}Q^{T}\left(s,x-z\right)f\left(z\right)dz.
\]
Using the boundary conditions we can extend the last equation for $x\geq T$
as
\[
a\frac{\partial}{\partial x}Q^{T}\left(s,x\right)=-\left(s+\lambda\right)Q^{T}\left(s,x\right)+\lambda\int_{-\infty}^{\infty}Q^{T}\left(s,x-z\right)f\left(z\right)dzdy-C_{T}\left(x\right),
\]
where $C_{T}\left(x\right)=\lambda\int_{-\infty}^{\infty}Q^{T}\left(s,x-z\right)f\left(z\right)dzI_{\left\{ x\geq T\right\} }$
and $I_{\{A\}}$ is the indicator of the event $A$. Using the factorization
identity and projection operation $\left[\cdot\right]_{J}$, $J\subset\left(-\infty,\infty\right)$,
which is defined for an absolutely integrable function $g\left(x\right)$
as
\[
\left[\int_{-\infty}^{\infty}e^{rx}g\left(x\right)dx\right]_{J}=\int_{J}e^{rx}g\left(x\right)dx,\,\left[\int_{-\infty}^{\infty}e^{rx}g\left(x\right)dx+C\right]_{\pm}^{0}=\mp\int_{0}^{\pm\infty}e^{rx}g\left(x\right)dx+C,
\]
from the integro--differential equation we can deduce the integral
transform of m.g.f. of the exit time through upper bound.
\begin{lem}
\label{lem:1-2} The integral transform of m.g.f. $Q^{T}\left(s,x\right)$
for the prelimit Kou process has the representation
\begin{multline}
\int_{0}^{\infty}e^{rx}Q^{T}\left(s,x\right)dx=s^{-1}\mathsf{E}e^{r\xi^{+}\left(\theta_{s}\right)}\left[\mathsf{E}e^{r\xi^{-}\left(\theta_{s}\right)}a\left(1-e^{rT}Q^{T}\left(s,T-0\right)\right)\right]_{+}^{0}+\\
+s^{-1}\mathsf{E}e^{r\xi^{+}\left(\theta_{s}\right)}\left[\mathsf{E}e^{r\xi^{-}\left(\theta_{s}\right)}\int_{0}^{\infty}e^{rx}\left(\int_{-\infty}^{0}Q^{T}\left(s,z\right)\lambda f\left(x-z\right)dz-C_{T}\left(x\right)\right)dx\right]_{+}^{0}.\label{eq:3.1-1}
\end{multline}
\end{lem}
\begin{proof} Since function $Q^{T}\left(s,x\right)$ has jump at
point $T$: $Q^{T}\!\left(s,T-0\right)$ $\neq Q^{T}\left(s,T\right)=0$,
we see that
\begin{gather*}
-ar\int_{0}^{\infty}e^{rx}Q^{T}\left(s,x\right)dx=a\left(1-e^{rT}Q^{T}\left(s,T-0\right)\right)+\int_{0}^{\infty}e^{rx}a\frac{\partial}{\partial x}Q^{T}\left(s,x\right)dx,\\
\int_{0}^{\infty}e^{rx}a\frac{\partial}{\partial x}Q^{T}\left(s,x\right)dx=-\left(s+\lambda\right)\int_{0}^{\infty}e^{rx}Q^{T}\left(s,x\right)dx+\\
+\int_{0}^{\infty}e^{rx}\int_{-\infty}^{\infty}Q^{T}\left(s,x-z\right)\lambda f\left(z\right)dzdx-\int_{0}^{\infty}e^{rx}C_{T}\left(x\right)dx.
\end{gather*}
Whence
\begin{multline*}
\left(s-k\left(r\right)\right)\int_{0}^{\infty}e^{rx}Q^{T}\left(s,x\right)dx=a\left(1-e^{rT}Q^{T}\left(s,T-0\right)\right)-\int_{0}^{\infty}e^{rx}C_{T}\left(x\right)dx+\\
+\int_{0}^{\infty}e^{rx}\int_{-\infty}^{0}Q^{T}\left(s,z\right)\lambda f\left(x-z\right)dzdx-\int_{-\infty}^{0}e^{rx}\int_{0}^{\infty}Q^{T}\left(s,z\right)\lambda f\left(x-z\right)dzdx,
\end{multline*}

Utilizing now the factorization identity and the projection operation yields
the equality
\begin{multline*}
s\int_{0}^{\infty}e^{rx}Q^{T}\left(s,x\right)dx=\mathsf{E}e^{r\xi^{+}\left(\theta_{s}\right)}\left[\mathsf{E}e^{r\xi^{-}\left(\theta_{s}\right)}a\left(1-e^{rT}Q^{T}\left(s,T-0\right)\right)\right]_{+}^{0}-\\
-\mathsf{E}e^{r\xi^{+}\left(\theta_{s}\right)}\left[\mathsf{E}e^{r\xi^{-}\left(\theta_{s}\right)}\int_{-\infty}^{0}e^{rx}\int_{0}^{\infty}Q^{T}\left(s,z\right)\lambda f\left(x-z\right)dzdx\right]_{+}^{0}+\\
+\mathsf{E}e^{r\xi^{+}\left(\theta_{s}\right)}\left[\mathsf{E}e^{r\xi^{-}\left(\theta_{s}\right)}\int_{0}^{\infty}e^{rx}\left(\int_{-\infty}^{0}Q^{T}\left(s,z\right)\lambda f\left(x-z\right)dz-C_{T}\left(x\right)\right)dx\right]_{+}^{0}.
\end{multline*}
The second term is zero and we obtain (\ref{eq:3.1-1}).
\end{proof}

Inverting (\ref{eq:3.1-1}) with respect to $r$ gives
\begin{multline}
sQ^{T}\left(s,x\right)=sP'_{+}\left(s,x\right)C_{+}\left(s\right)-aQ^{T}\left(s,T-0\right)\times\\
\times\left(\int_{-T}^{\min\left\{ x-T,0\right\} }P'_{+}\left(s,x-y-T\right)P'_{-}\left(s,y\right)dy+p_{-}\left(s\right)P'_{+}\left(s,x-T\right)I_{\{x\geq T\}}\right)+\\
+\int_{0}^{x}\int_{-\infty}^{0}\int_{-\infty}^{0}Q^{T}\left(s,u\right)\lambda f\left(x-y-z-u\right)dudP_{-}\left(s,z\right)dP_{+}\left(s,y\right)-\\
-\int_{0}^{x}\int_{-\infty}^{0}C_{T}\left(x-y-z\right)dP_{-}\left(s,z\right)dP_{+}\left(s,y\right).\label{eq:3.1}
\end{multline}

Observe, that for $u\leq0:$ $Q^{T}\left(s,u\right)=1$, and for $u>0:$
$\lambda f\left(u\right)=\lambda pce^{-cu}$, hence
\begin{gather*}
\int_{-\infty}^{0}Q^{T}\left(s,u\right)\lambda f\left(x-y-z-u\right)du=\int_{-\infty}^{0}\lambda f\left(x-y-z-u\right)du=\lambda pe^{-c\left(x-y-z\right)},\\
C_{T}\left(x-y-z\right)=\lambda pe^{-c\left(x-y-z\right)}\left(\int_{0}^{T}Q^{T}\left(s,u\right)ce^{cu}du+1\right)I_{\left\{ x-y-z\geq T\right\} }.
\end{gather*}
Denote $C_{0}\left(T\right)=\int_{0}^{T}Q^{T}\left(s,u\right)ce^{cu}du+1$
and $C_{1}\left(T\right)=s^{-1}aQ^{T}\left(s,T-0\right)$, then from
(\ref{eq:3.1}) it follows that
\begin{multline}
Q^{T}\left(s,x\right)=P\left\{ \xi^{+}\left(\theta_{s}\right)\geq x\right\} -C_{1}\left(T\right)\int_{-T}^{\min\left\{ x-T,0\right\} }P'_{+}\left(s,x-y-T\right)dP_{-}\left(s,y\right)-\\
-C_{0}\left(T\right)\int_{0}^{x}\int_{-\infty}^{x-y-T}s^{-1}\lambda pe^{-c\left(x-y-z\right)}dP_{-}\left(s,z\right)dP_{+}\left(s,y\right).\label{eq:3.2}
\end{multline}
Formula (\ref{eq:3.2}) determines the m.g.f. of exit time through
the upper bound for the prelimit Kou process.
\begin{thm}
\label{thm:1} For a Kou process the m.g.f. of exit time from the interval
$\left(x-T,x\right)$, $0\leq x\leq T$ through upper bound has the
representation
\begin{multline}
Q^{T}\left(s,x\right)=P\left\{ \xi^{+}\left(\theta_{s}\right)\geq x\right\} -C_{1}\left(T\right)\int_{-T}^{x-T}P'_{+}\left(s,x-y-T\right)P'_{-}\left(s,y\right)dy-\\
-C_{0}\left(T\right)\int_{0}^{x}\int_{-\infty}^{x-y-T}s^{-1}\lambda pe^{-c\left(x-y-z\right)}P'_{-}\left(s,z\right)P'_{+}\left(s,y\right)dzdy.\label{eq:3.3}
\end{multline}
where densities $P'_{\pm}\left(s,x\right)$ are defined by (\ref{eq:2.5a})
-- (\ref{eq:2.5b}), $C_{0}\left(T\right)$ and $C_{1}\left(T\right)$
satisfy equations $C_{0}\left(T\right)=\int_{0}^{T}Q^{T}\left(s,u\right)ce^{cu}du+1$
and $Q^{T}\left(s,T\right)=0$.

For the m.g.f. of exit time through the lower bound the next formula holds
\begin{multline}
Q_{T}\left(s,x\right)=P\left\{ \xi^{-}\left(\theta_{s}\right)\leq x-T\right\} -C^{1}\left(T\right)\int_{x}^{T}P'_{-}\left(s,x-y\right)P'_{+}\left(s,y\right)dy-\\
-C^{0}\left(T\right)\int_{x-T}^{0}\int_{x-y}^{\infty}s^{-1}\lambda qe^{b\left(x-y-z-T\right)}P'_{+}\left(s,z\right)P'_{-}\left(s,y\right)dzdy.\label{eq:3.3a}
\end{multline}
where $C^{0}\left(T\right)$, $C^{1}\left(T\right)$ obey equations
$C^{0}\left(T\right)=\int_{0}^{T}Q_{T}\left(s,u\right)be^{-b\left(u-T\right)}du+1$
and $Q_{T}\left(s,0\right)=0$.
\end{thm}
\begin{proof} Taking into account that as $n\rightarrow\infty$
the distribution of killed extrema and the distribution of exit
time from the interval for the prelimit process tend to the distributions
of the corresponding functionals of Kou process: $P_{\pm}^{n}\left(s,x\right)\rightarrow P{}_{\pm}\left(s,x\right)$,
$Q_{n}^{T}\left(s,x\right)\rightarrow Q^{T}\left(s,x\right)$, we
have that $C_{0}^{n}\left(T\right)\rightarrow C_{0}\left(T\right)$
and $C_{1}^{n}\left(T\right)\rightarrow C_{1}\left(T\right)$. Moreover,
$\mathsf{P}\left\{ \xi_{n}^{-}(\theta_{s})=0\right\} \rightarrow0$
as $n\rightarrow\infty$ gives (\ref{eq:3.3}) when combined with (\ref{eq:3.2}).

To derive the m.g.f. for the exit time through the lower bound we use
the fact that $Q_{T}\left(s,x\right)=Q_{1}^{T}\left(s,T-x\right)$,
where $Q_{1}^{T}\left(s,x\right)$ is the m.g.f. of exit time
through the upper bound for the dual process $\xi_{1}\left(t\right)=-\xi\left(t\right)$.
\end{proof}
Denote the integrals in (\ref{eq:3.3}) by
\begin{gather*}
J_{1}\left(s,x,T\right)=\int_{-T}^{x-T}P'_{+}\left(s,x-y-T\right)P'_{-}\left(s,y\right)dy,\\
J_{2}\left(s,x,T\right)=\int_{0}^{x}\int_{-\infty}^{x-y-T}s^{-1}\lambda pe^{-c\left(x-y-z\right)}P'_{-}\left(s,z\right)P'_{+}\left(s,y\right)dzdy,
\end{gather*}
then
\[
Q^{T}\left(s,x\right)=P\left\{ \xi^{+}\left(\theta_{s}\right)\geq x\right\} -C_{1}\left(T\right)J_{1}\left(s,x,T\right)-C_{0}\left(T\right)J_{2}\left(s,x,T\right).
\]
To find $C_{0}\left(T\right)$ and $C_{1}\left(T\right)$ use the boundary
conditions for $Q^{T}\left(s,x\right)$: $Q^{T}\left(s,T\right)=0$
and $\int_{0}^{T}Q^{T}\left(s,u\right)ce^{cu}du+1=C_{0}\left(T\right)$. Then
\[
C_{0}\left(T\right)=\frac{\left(1+\tilde{J}_{0}\right)J_{1}-\tilde{J}_{1}J_{0}}{J_{1}\left(1+\tilde{J}_{2}\right)-J_{2}\tilde{J}_{1}},C_{1}\left(T\right)=\frac{\left(1+\tilde{J}_{2}\right)J_{0}-J_{2}\left(1+\tilde{J}_{0}\right)}{J_{1}\left(1+\tilde{J}_{2}\right)-J_{2}\tilde{J}_{1}},
\]
where $J_{0}=P\left\{ \xi^{+}\left(\theta_{s}\right)\geq T\right\} ,$
$J_{1}=J_{1}\left(s,T,T\right)$, $J_{2}=J_{2}\left(s,T,T\right)$,
$\tilde{J}_{0}=\int_{0}^{T}\overline{P}_{+}\left(s,u\right)ce^{cu}du$,
$\tilde{J}_{1,2}=\int_{0}^{T}J_{1,2}\left(s,u,T\right)ce^{cu}du$.

\section{The joint distribution of two-boundary functionals }
We include into consideration the value of overshoot through a
boundary at the exit moment from the interval by a Kou process:
\[
\gamma_{T}\left(x\right)=(\xi\left(\tau\left(x,T\right)\right)-x)I_{A_{+}(x)}+(x-T-\xi(\tau(x,T))I_{A_{-}(x)}.
\]
Appling similar arguments as in Section \ref{sub:M.g.f. of the exit}
we can establish the following integral equation for the joint m.g.f.
of $\left\{ \tau\left(x,T\right),\gamma_{T}\left(x\right)\right\} $
$\left(\mathrm{Im}(\alpha)=0\right)$
\begin{multline}
\mathsf{E}\left[e^{-s\tau\left(x,T\right)+i\alpha\gamma_{T}\left(x\right)},A_{+}\left(x\right)\right]=\mathsf{E}\left[e^{-s\tau^{+}\left(x\right)+i\alpha\gamma^{+}\left(x\right)},\tau^{+}\left(x\right)<\infty\right]-\\
-C_{1}\left(T,\alpha\right)J_{1}(s,x,T)-C_{0}\left(T,\alpha\right)J_{2}(s,x,T),\label{eq:3.4}
\end{multline}
where $C_{0}\left(T,\alpha\right)$ and $C_{1}\left(T,\alpha\right)$
obey equations $\mathsf{E}\left[e^{-s\tau\left(T,T\right)+i\alpha\gamma_{T}\left(T\right)},A_{+}\left(T\right)\right]=0$
and $C_{0}\left(T,\alpha\right)=\int_{0}^{T}\mathsf{E}\left[e^{-s\tau\left(u,T\right)+i\alpha\gamma_{T}\left(u\right)},A_{+}\left(u\right)\right]ce^{cu}du+c(c-i\alpha)^{-1}$.

Denote
\begin{gather*}
V_{0}=\mathsf{E}\!\left[e^{-s\tau^{+}\left(T\right)},\gamma^{+}\left(T\right)=0,\tau^{+}\left(T\right)<\infty\right]\!,V_{>}=\mathsf{E}\!\left[e^{-s\tau^{+}\left(T\right)},\gamma^{+}\left(T\right)>0,\tau^{+}\left(T\right)<\infty\right]\!,\\
\tilde{V}_{0}=\int_{0}^{T}\mathsf{E}\left[e^{-s\tau^{+}\left(u\right)},\gamma^{+}\left(u\right)=0,\tau^{+}\left(u\right)<\infty\right]ce^{cu}du,\\
\begin{gathered}\tilde{V}_{>0}=\int_{0}^{T}\mathsf{E}\left[e^{-s\tau^{+}\left(u\right)},\gamma^{+}\left(u\right)>0,\tau^{+}\left(u\right)<\infty\right]ce^{cu}du,\end{gathered}
\end{gather*}
then from (\ref{eq:3.4}) the matrix form for the joint m.g.f. follows
\begin{multline*}
\mathsf{E}\left[e^{-s\tau\left(x,T\right)+i\alpha\gamma_{T}\left(x\right)},A_{+}\left(x\right)\right]=\mathsf{E}\left[e^{-s\tau^{+}\left(x\right)},\gamma^{+}\left(x\right)=0,\tau^{+}\left(x\right)<\infty\right]+\\
+\mathsf{E}\left[e^{-s\tau^{+}\left(x\right)},\gamma^{+}\left(x\right)>0,\tau^{+}\left(x\right)<\infty\right]\frac{c}{c-i\alpha}-\\
-\left(J_{1}\left(s,x,T\right);J_{2}\left(s,x,T\right)\right)\begin{pmatrix}J_{1} & J_{2}\\
\tilde{J}_{1} & 1+\tilde{J}_{2}
\end{pmatrix}^{-1}\begin{pmatrix}V_{0}+\frac{c}{c-i\alpha}V_{>}\\
\tilde{V}_{0}+\frac{c}{c-i\alpha}\tilde{V}_{>}
\end{pmatrix}.
\end{multline*}
We thus get the next assertion.
\begin{thm}
\label{thm:2} For the joint m.g.f. of the first exit time from
$\left(x-T,x\right)$, $0\leq x\leq T$ by a Kou process through the upper
bound and the corresponding overshoot the following equalities are
hold
\begin{multline}
\mathsf{E}\left[e^{-s\tau\left(x,T\right)},\gamma_{T}\left(x\right)=0,A_{+}\left(x\right)\right]=\mathsf{E}\left[e^{-s\tau^{+}\left(x\right)},\gamma^{+}\left(x\right)=0,\tau^{+}\left(x\right)<\infty\right]-\\
-\left(J_{1}\left(s,x,T\right);J_{2}\left(s,x,T\right)\right)\begin{pmatrix}J_{1} & J_{2}\\
\tilde{J}_{1} & 1+\tilde{J}_{2}
\end{pmatrix}^{-1}\begin{pmatrix}V_{0}\\
\tilde{V}_{0}
\end{pmatrix}\label{eq:3.5a}
\end{multline}
and
\begin{multline}
\mathsf{E}\!\left[e^{-s\tau\left(x,T\right)+i\alpha\gamma_{T}\left(x\right)},\gamma_{T}\left(x\right)>0,A_{+}\left(x\right)\right]\!=\!\frac{c}{c-i\alpha}\!\left(\!\mathsf{E}\!\left[e^{-s\tau^{+}\left(x\right)},\gamma^{+}\left(x\right)>0,\tau^{+}\left(x\right)<\infty\right]\!-\right.\\
-\left.\left(J_{1}\left(s,x,T\right);J_{2}\left(s,x,T\right)\right)\begin{pmatrix}J_{1} & J_{2}\\
\tilde{J}_{1} & 1+\tilde{J}_{2}
\end{pmatrix}^{-1}\begin{pmatrix}V_{>}\\
\tilde{V}_{>}
\end{pmatrix}\right).\label{eq:3.5b}
\end{multline}
\end{thm}
Formulas (\ref{eq:3.5a}) -- (\ref{eq:3.5b}) imply that the exit
time and the corresponding overshoot are conditionally independent
and the overshoot is conditionally memoryless given that $\left\{ \gamma_{T}\left(x\right)>0\right\} $.

\subsection{Density function of the process before exit from the interval}
To find the m.g.f. for the process before the exit from the interval use
the Pecherskii identity (see~\cite[Th. 4.3]{Husak2011engl}):
\begin{multline}
\mathsf{E}\left[e^{i\alpha\xi\left(\theta_{s}\right)},\tau\left(x,T\right)>\theta{}_{s}\right]=\\
=\mathsf{E}e^{i\alpha\xi^{+}(\theta_{s})}\left[\mathsf{E}e^{i\alpha\xi^{-}(\theta_{s})}\left(1-e^{i\alpha x}\mathsf{E}\left[e^{-s\tau\left(x,T\right)+i\alpha\gamma_{T}\left(x\right)},A_{+}\left(x\right)\right]\right)\right]_{\left[x-T,\infty\right)}=\\
=\mathsf{E}e^{i\alpha\xi^{+}(\theta_{s})}\left[\mathsf{E}e^{i\alpha\xi^{-}(\theta_{s})}\right]_{\left[x-T,\infty\right)}-\\
-\mathsf{E}e^{i\alpha\xi^{+}(\theta_{s})}\left[\mathsf{E}e^{i\alpha\xi^{-}(\theta_{s})}e^{i\alpha x}\mathsf{E}\left[e^{-s\tau\left(x,T\right)},\gamma_{T}\left(x\right)=0,A_{+}\left(x\right)\right]\right]_{\left[x-T,\infty\right)}-\\
-\mathsf{E}e^{i\alpha\xi^{+}(\theta_{s})}\left[\mathsf{E}e^{i\alpha\xi^{-}(\theta_{s})}e^{i\alpha x}\mathsf{E}\left[e^{-s\tau\left(x,T\right)+i\alpha\gamma_{T}\left(x\right)},\gamma_{T}\left(x\right)>0,A_{+}\left(x\right)\right]\right]_{\left[x-T,\infty\right)}.\label{eq:3.9}
\end{multline}
For the first term in (\ref{eq:3.9}) we have
$$
\mathsf{E}e^{i\alpha\xi^{+}(\theta_{s})}\left[\mathsf{E}e^{i\alpha\xi^{-}(\theta_{s})}\right]_{\left[x-T,\infty\right)}=
\int_{x-T}^{\infty}e^{i\alpha z}\int_{x-T}^{\min\left\{ 0,z\right\} }P'_{+}\left(s,z-y\right)P'_{-}\left(s,y\right)dydz,
$$
for the second
\begin{multline*}
\mathsf{E}e^{i\alpha\xi^{+}(\theta_{s})}\left[\mathsf{E}e^{i\alpha\xi^{-}(\theta_{s})}e^{i\alpha x}\mathsf{E}\left[e^{-s\tau\left(x,T\right)},\gamma_{T}\left(x\right)=0,A_{+}\left(x\right)\right]\right]_{\left[x-T,\infty\right)}=\\
=\int_{x-T}^{\infty}e^{i\alpha z}\int_{x-T}^{\min\left\{ z,x\right\} }P'_{+}\left(s,z-y\right)P'_{-}\left(s,y\right)dy\mathsf{E}\left[e^{-s\tau\left(x,T\right)},\gamma_{T}\left(x\right)=0,A_{+}\left(x\right)\right],
\end{multline*}
and for the third we find that
\begin{multline*}
\mathsf{E}e^{i\alpha\xi^{+}(\theta_{s})}\left[\mathsf{E}e^{i\alpha\xi^{-}(\theta_{s})}e^{i\alpha x}\mathsf{E}\left[e^{-s\tau\left(x,T\right)+i\alpha\gamma_{T}\left(x\right)},\gamma_{T}\left(x\right)>0,A_{+}\left(x\right)\right]\right]_{\left[x-T,\infty\right)}=\\
=\int_{x-T}^{\infty}e^{i\alpha z}\int_{x-T}^{z}P'_{+}\left(s,z-v\right)\int_{-\infty}^{\min\left\{ x,v\right\} }ce^{-c\left(v-y\right)}P'_{-}\left(s,y\right)dydvdz\times\\
\times\mathsf{E}\left[e^{-s\tau\left(x,T\right)},\gamma_{T}\left(x\right)>0,A_{+}\left(x\right)\right].
\end{multline*}
Hence, inverting (\ref{eq:3.9}) with respect to $\alpha$ we can deduce
the next statement.
\begin{thm}
\label{thm:3} The density function of killed Kou process until the
exit time from $\left(x-T,x\right)$ has the representation
\begin{multline}
h_{s}\left(T,x,z\right)=\frac{\partial}{\partial z}P\left\{ \xi\left(\theta_{s}\right)<z,\tau\left(x,T\right)>\theta_{s}\right\} =\int_{x-T}^{\min\left\{ 0,z\right\} }P'_{+}\left(s,z-y\right)P'_{-}\left(s,y\right)dy-\\
-\int_{x-T}^{\min\left\{ z,x\right\} }P'_{+}\left(s,z-y\right)P'_{-}\left(s,y\right)dy\mathsf{E}\left[e^{-s\tau\left(x,T\right)},\gamma_{T}\left(x\right)=0,A_{+}\left(x\right)\right]-\\
-\int_{x-T}^{z}P'_{+}\left(s,z-v\right)\int_{-\infty}^{\min\left\{ x,v\right\} }ce^{-c\left(v-y\right)}P'_{-}\left(s,y\right)dydv\mathsf{E}\left[e^{-s\tau\left(x,T\right)},\gamma_{T}\left(x\right)>0,A_{+}\left(x\right)\right].\label{eq:3.6}
\end{multline}
\end{thm}
Density $h_{s}\left(T,x,y\right)$ determines the joint distribution
of $\left\{ \xi^{-}\left(\theta_{s}\right),\xi\left(\theta_{s}\right),\xi^{+}\left(\theta_{s}\right)\right\} $
(for details see \cite{Kadankov2005engl}).

\noindent{\bf{Example}.} We consider the two sets of parameters (see Tabl.~\ref{tabl1}).
The main difference is that $m=\mathsf{E}\xi(1)$ have different signs:
for the first case $m<0$, and for the second $m>0$. The sign of
the expectation determines the asymptotic behavior of the roots of
cumulant equation.\par From the cumulant equation
\[
ar+r^{2}\frac{\sigma^{2}}{2}+\lambda r\left(\frac{p}{c-r}-\frac{q}{b+r}\right)=s
\]
we find the approximated values of roots, then using~(\ref{eq:2.5a})--~(\ref{eq:2.5b})
 we obtain the densities of killed extrema (see Tabl.~\ref{tabl1}).
\begin{table}[ht]\label{tabl1}
\centering
\begin{tabular}{|l|l|}
\hline
 & Case 1)\tabularnewline
\hline
Parameters  & $a=-3.05,p=0.75,\lambda=4,c=2,b=8,\sigma=0.5,s=4/45$\tabularnewline
Roots & $r_{2}\approx8.2084,r_{1}\approx0.0515,\rho_{1}\approx1.0000,\rho_{2}\approx13.4599$\tabularnewline
Densities  & $P'_{-}(s,x)\approx0.0515e^{0.0515x}+0.0014e^{8.2084x},x<0$\tabularnewline
 & $P'_{+}(s,x)\approx6.1898e^{-13.4599x}+0.5401e^{-1.0000x},x>0$ \tabularnewline
\hline
 & Case 2)\tabularnewline
\hline
Parameters  & $a=1,p=0.2,\lambda=6,c=2,b=8,\sigma=2,s=1$\tabularnewline
Roots & $r_{2}\approx8.6470,r_{1}\approx1.3654,\rho_{1}\approx0.5504,\rho_{2}\approx2.4621$\tabularnewline
Densities  & $P'_{-}(s,x)=1.3447e^{1.3654x}+0.1311e^{8.6470x},x<0$\tabularnewline
 & $P'_{+}(s,x)=0.1638e^{-2.4621x}+0.5138e^{-0.5504x},x>0$\tabularnewline
\hline
\end{tabular}
\caption{Density functions of killed extrema}
\end{table}

Substitution these densities in~(\ref{eq:3.3}) --
(\ref{eq:3.3a}) yields the representations for $Q_{T}\left(s,x\right)$,
$Q^{T}\left(s,x\right)$ and their sum $Q\left(s,x,T\right)=\mathsf{E}e^{-s\tau\left(x,T\right)}$
(see Fig. \ref{pic.QTtsx}). Due to complexity we omit the corresponding
expressions.
\begin{figure}[ht]
\centering \includegraphics[width=13cm]{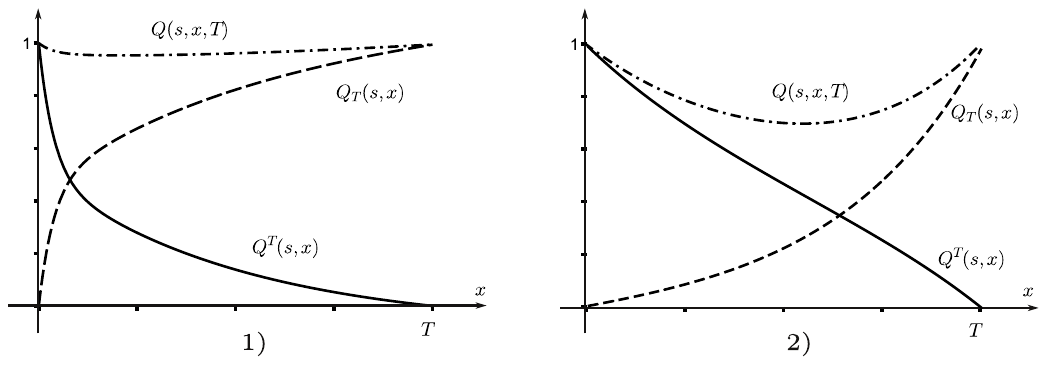}\caption{M.g.f. for the exit time from $\left(x-T,x\right)$
}\label{pic.QTtsx}
\end{figure}

\end{document}